\documentclass[11pt,leqno]{amsart}
\usepackage{amssymb,amsthm}
\usepackage{amsmath}
\usepackage{setspace}
\usepackage{dcpic,pictexwd}
\usepackage[all]{xy}
\usepackage[pdftex]{graphicx}
\usepackage{mathrsfs}
\usepackage[pdftex]{hyperref}
\usepackage{bbm}

\theoremstyle{definition}
 \newtheorem{definition}{Definition}[section]

\theoremstyle{plain}

\theoremstyle{plain}
 \newtheorem{theorem}[definition]{Theorem}
 
\theoremstyle{definition}
 \newtheorem{example}[definition]{Example}

\theoremstyle{plain}

\theoremstyle{plain}
 \newtheorem{corollary}[definition]{Corollary}

\theoremstyle{remark}
 \newtheorem{remark}[definition]{Remark}

\theoremstyle{definition}

\theoremstyle{plain}

\newcommand{\Ext}{\mathrm{Ext}}
\newcommand{\End}{\mathrm{End}}

\newcommand{\Hom}{\mathrm{Hom}}

\newcommand{\Ca}{\mathcal{C}}
\newcommand{\Fun}{\mathrm{F}}

\newcommand{\Def}{\mathrm{Def}}
\newcommand{\Sets}{\mathrm{Sets}}

\newcommand{\Ob}{\mathrm{Ob}}

\newcommand{\Z}{\mathbb{Z}}
\newcommand{\SEnd}{\underline{\End}}
\newcommand{\A}{\Lambda}
\newcommand{\G}{\Gamma}

\newcommand{\m}{\mathfrak{m}}
\renewcommand{\k}{\Bbbk}

\newcommand{\invlim}{\varprojlim}

\hypersetup{
    bookmarks=true,         
    unicode=false,          
    pdftoolbar=true,        
    pdfmenubar=true,        
    pdffitwindow=false,     
    pdfstartview={FitH},    
    pdftitle={Universal deformation rings and non-negative derived equivalences},    
    pdfauthor={Shengyong Pan and Jose A. Velez-Marulanda},     
    pdfsubject={},   
    pdfcreator={Shengyong Pan and Jose A. Velez-Marulanda},   
    pdfproducer={Shengyong Pan and Jose A. Velez-Marulanda}, 
    pdfkeywords={Universal deformation rings} {Singular Equivalences of Morita Type with Level}{Finitely generated Gorenstein-projective modules}, 
    pdfnewwindow=true,      
    colorlinks=true,       
    linkcolor=red,          
    citecolor=blue,        
    filecolor=magenta,      
    urlcolor=blue           
}

\title[Universal deformation rings and stable equivalences]{On universal deformation rings and stable equivalences of Gorenstein-projective modules} 

\author{Shengyong Pan}
\address{School of Mathematics and Statistics, Beijing Jiaotong University,
Beijing, 100044, P. R. China}
\email{shypan@bjtu.edu.cn}

\author{Jos\'e A. V\'elez-Marulanda}
\address{Department of Data Science, Valdosta State University, Valdosta, GA, U.S.A.}
\email{javelezmarulanda@valdosta.edu}
\address{Facultad de Matem\'aticas e Ingenier\'{\i}as, Fundaci\'on Universitaria Konrad Lorenz, Bogot\'a, Colombia}
\email{josea.velezm@konradlorenz.edu.co}

\keywords{Universal deformation rings \and stable endomorphism rings \and Gorenstein-projective modules \and Non-negative derived equivalences}
\begin{document}
\renewcommand{\labelenumi}{\textup{(\roman{enumi})}}
\renewcommand{\labelenumii}{\textup{(\roman{enumi}.\alph{enumii})}}
\numberwithin{equation}{section}


\begin{abstract}

Let $\k$ be a field and let $\A$ and $\G$ finite dimensional $\k$-algebras. Assume that ${_\G}X_\A$ and ${_\A}Y_\G$ are bimodules that define a singular equivalence of Morita type with level (in the sense of Z. Wang) between $\A$ and $\G$ and which also induce an equivalence between the stable categories of finitely generated Gorenstein-projective modules $\A\textup{-\underline{Gproj}}$ and $\G\textup{-\underline{Gproj}}$. We prove that if $V$ is an indecomposable object in $\A\textup{-\underline{Gproj}}$ with $\SEnd_\A(V)\cong \k$, then $X\otimes_\A V$ is an object in $\G\textup{-\underline{Gproj}}$ such that $\SEnd_\G(X\otimes_\A V)\cong \k$ and the universal deformation rings (in the sense of F.M. Bleher and the second author) $R(\A,V)$ and $R(\G, X\otimes_\A V)$ are isomorphic. This result generalizes the one obtained by the second author assuming that $\A$ and $\G$ are Gorenstein $\k$-algebras.
\end{abstract}
\subjclass[2020]{16G10 \and 16G20 \and 16G50}
\maketitle

\section{Introduction}\label{int}

Throughout this article, we assume that $\k$ is a fixed field of arbitrary characteristic and all our modules are finite dimensional over $\k$. For all finite dimensional $\k$-algebras $\A$, we denote by $\mathcal{D}^b(\A\textup{-mod})$ its bounded derived category and by $\mathcal{D}_{sg}(\A\textup{-mod})$ its {\it singularity category}, i.e. $\mathcal{D}_{sg}(\A\textup{-mod})$ is the Verdier quotient  $\mathcal{D}^b(\A\textup{-mod})/\mathcal{K}^b(\A\textup{-proj})$ where $\mathcal{K}^b(\A\textup{-proj})$ is the category of {\it perfect complexes} over $\A$ (see e.g. \cite[\S 6.2]{krause4}).  In this article, we assume that for all finite dimensional $\k$-algebras $\A$ and all integers $\ell\geq 0$, the $\ell$-th zyzygy of $\A$ as a $\A$-$\A$-bimodule, namely $\Omega_{\A\text{-}\A}^\ell\A$ is indecomposable as a $\A$-$\A$-bimodule.  The following definition is due to Z. Wang (see \cite[Def. 2.1]{wang}). 

\begin{definition}\label{defi:3.2}
Let $\A$ and $\Gamma$ be finite dimensional $\k$-algebras, and let $X$ be a $\Gamma$-$\A$-bimodule and $Y$ a $\A$-$\Gamma$-bimodule. We say that the pair $({_\Gamma}X_\A,{_\A}Y_\Gamma)$ induces a {\it singular equivalence of Morita type with level} $\ell\geq 0$ between $\A$ and $\Gamma$ and say that $\A$ and $\Gamma$ are {\it singularly equivalent of Morita type with level $\ell$} if the following conditions are satisfied:
\begin{enumerate}
\item $X$ is projective as a left $\Gamma$-module and as a right $\A$-module;
\item $Y$ is projective as a left $\A$-module and as a right $\Gamma$-module; 
\item $X\otimes_\A Y\cong \Omega_{\Gamma^e}^\ell \Gamma$ in $\Gamma^e$-\underline{mod};
\item $Y\otimes_\Gamma X\cong \Omega_{\A^e}^\ell \A$ in $\A^e$-\underline{mod}.  
\end{enumerate}
\end{definition}

It follows that under the situation of Definition \ref{defi:3.2}, the bimodules $X$ and $Y$ induce equivalences of singularity categories $X\otimes_\A-:\mathcal{D}_{sg}(\A\textup{-mod})\to \mathcal{D}_{sg}(\G\textup{-mod})$ and $Y\otimes_\G-: \mathcal{D}_{sg}(\G\textup{-mod}) \to \mathcal{D}_{sg}(\A\textup{-mod})$ which are quasi-inverse of each other. These singular equivalances of Morita type with level generalize stable equivalences of Morita type (in the sense of \cite{broue}) as well as singular equivalences of Morita type (in the sense of \cite{chensun,zhouzimm}). Assume that $\A$ is a finite dimensional $\k$-algebra and that $V$ is a left $\A$-module. Recall that $V$ is said to be {\it Gorenstein-projective} if it satisfies the following conditions:
\begin{enumerate}
\item $V$ is {\it reflexive}, i.e. there is an isomorphism of left $\A$-modules $V\cong \Hom_\A(\Hom_\A(V,\A),\A)$;
\item $\Ext_\A^i(V,\A) = 0 = \Ext_\A^i(\Hom_\A(V,\A),\A)$ for all $i\in \Z$. 
\end{enumerate}
We denote by $\A$-Gproj the category of Gorenstein-projective left $\A$-modules, and by $\A\textup{-\underline{Gproj}}$ its stable category.  It is well-known that $\A$-Gproj is a Frobenius category in the sense of \cite[Chap. I, \S 2.1]{happel} and consequently,  $\A\textup{-\underline{Gproj}}$ is a triangulated category (in the sense of \cite{verdier}). Moreover, if $V$ is non-projective Gorenstein-projective, then for all $i\geq 0$, the $i$-th syzygy $\Omega_\A^iV$ is also non-projective Gorenstein-projective, and $\Omega_\A$ induces an autoequivalence $\Omega_\A:\A\textup{-\underline{Gproj}}\to \A\textup{-\underline{Gproj}}$ (see  \cite[Chap. I, \S2.2]{happel}).  Recall that a $\k$-algebra $\A$ is said to be Gorenstein if $\A$ has finite injective dimension as a left $\A$-module and as a right $\A$-module. It follows from a seminal result due to R. O. Buchweitz (see \cite{buchweitz} and e.g. \cite[Thm. 6.2.5]{krause4}) that  if $\A$ is Gorenstein then the triangulated categories $\mathcal{D}_{sg}(\A\textup{-mod})$  and $\A\textup{-\underline{Gproj}}$ are equivalent. 
Assume that $V$ is Gorenstein-projective such that  $\SEnd_\A(V)\cong \k$. Then it follows from \cite[Thm. 1.2(ii)]{bekkert-giraldo-velez} that $V$ has a universal deformation ring $R(\A,V)$ (in the sense of \cite{blehervelez}) which is a local complete commutative Noetherian $\k$-algebra with residue field isomorphic to $\k$. It was proved by the second author in \cite[Thm. 1.2 (i)]{velez4} that under this situation, $\SEnd_\A(\Omega V)\cong \k$ and the universal deformation rings $R(\A,V)$ and $R(\A,\Omega V)$ are isomorphic. 
Assume that $\G$ as another finite dimensional $\k$-algebra and that  ${_\G}X_\A$ and ${_\A}Y_\G$ are bimodules that induce a singular equivalence of Morita type with level (as in Definition \ref{defi:3.2}) between $\A$ and $\G$. It follows from \cite[Thm. 1.2 (ii)]{velez4} that if $\A$ and $\G$ are both Gorenstein then the left $\G$-module $X\otimes_\A V$ is also Gorenstein-projective with $\SEnd_\G(X\otimes_\A V)\cong \k$ and the universal deformation ring $R(\G, X\otimes_\A V)$ is isomorphic to $R(\A,V)$. To prove this result, the second author used the fact that since $\A$ and $\G$ are both Gorenstein,  the bimodules $X$ and $Y$ induce equivalences of triangulated categories $X\otimes_\A-:\A\textup{-\underline{Gproj}}\to \G\textup{-\underline{Gproj}}$ and $Y\otimes_\G-:\G\textup{-\underline{Gproj}}\to \G\textup{-\underline{Gproj}}$ that are quasi-inverse of each other (see \cite[Prop. 3.7]{skart}). This allowed him to provide immediate applications of universal deformation rings to Morita and triangular matrix $\k$-algebras, to singular equivalences induced by homological epimorphisms (in the sense of \cite{geigle}), and to $2$-recollements of triangulated categories (in the sense of \cite{qin2}) all this for Gorenstein algebras. However, it follows from \cite[Prop. 4.5]{wang} that the hypothesis of $\A$ and $\G$ being both Gorenstein can be replaced for that of $\Hom_\G(X,\G)$ and $\Hom_\A(Y,\A)$ have both finite projective dimensions as a left $\A$-module and a left $\G$-module, respectively. Moreover,  it also follows from \cite[Prop. 4.6]{wang} that if there exists a derived equivalence between $\mathcal{D}^b(\A\textup{-mod})$ and $\mathcal{D}^b(\G\textup{-mod})$, then there exists a pair a bimodules ${_\G}X_\A$ and ${_\A}Y_\G$ that induce a singular equivalence of Morita type with level between $\A$ and $\G$ as in Definition \ref{defi:3.2}, which further induce an equivalence between the triangulated categories $\A\textup{-\underline{Gproj}}$ and $\G\textup{-\underline{Gproj}}$.  


Assume further that $\A$ and $\G$ have no semisimple direct summands and that ${_\G}X_\A$ and ${_\A}Y_\G$ are bimodules that induce an stable equivalence of Morita type between $\A$ and $\G$ such that neither $X$ and $Y$ have projective bimodules as direct summands. By \cite[Lemma 4.1]{chen-pan-xi}, there are isomorphisms of bimodules $X\cong \Hom_\A(Y,\A)$ and $Y\cong \Hom_\G(X,\G)$, hence $\Hom_\A(Y,\A)$ and $\Hom_\G(X,\G)$ are projective as left $\G$-module and $\A$-module, respectively. Since as mentioned above, stable equivalences of Morita type are singular equivalences of Morita type with level, we obtain by using \cite[Prop. 4.5]{wang} that $X$ and $Y$ induce an equivalence between $\A\textup{-\underline{Gproj}}$ and $\G\textup{-\underline{Gproj}}$. 

For what is next, we refer the reader to see e.g \cite[Def. 2.2]{gao-ma-wang} for the definition of {\it Morita context algebras}.  Let $\Psi$ be the Morita context algebra
\begin{align}\label{mc1}
\Psi = \begin{pmatrix} \G & Q\\ W& C \end{pmatrix}\
\end{align}
 and which satisfies the following conditions:
 \begin{itemize}
 \item[(MC1):] $C$ is a finite dimensional $\k$-algebra that has finite projective dimension as a $C$-$C$-bimodule;
 \item[(MC2):] $W$ is a $C$-$\G$-bimodule with finite projective dimension as a bimodule, and it is projective as a one-sided module;
 \item[(MC3):] $Q$ is a $\G$-$C$-bimodule with finite projective dimension as a bimodule, and it is projective as a one-sided module. 
 \end{itemize}
 
 It follows from the main result in \cite{gao-ma-wang} that if $\A$ and $\G$ are singularly equivalent of Morita type with level $\ell$ (as in Definition \ref{defi:3.2}), and if $\Xi$ is the Morita context algebra 
 \begin{align}\label{mc2}
 \Xi = \begin{pmatrix} \A & Y\otimes_\G Q\\ W\otimes_\G X& C \end{pmatrix},
 \end{align}
 then there exists a pair of bimodules $({_\Psi}\widetilde{X}_{\Xi}, {_\Xi}\widetilde{Y}_{\Psi})$ that induce a singular equivalence of Morita type with level $\ell$ between $\Psi$ and $\Xi$. Moreover, if we again assume that $\Hom_\G(X,\G)$ and $\Hom_\A(Y,\A)$ have finite projective dimensions as a left $\A$-module and a left $\G$-module, respectively, then the bimodules $\widetilde{X}$ and $\widetilde{Y}$ induce equivalences of triangulated categories $\widetilde{X}\otimes_\Xi-:\Xi\textup{-\underline{Gproj}}\to \Psi\textup{-\underline{Gproj}}$ and $\widetilde{Y}\otimes_\Psi-:\Psi\textup{-\underline{Gproj}}\to \Xi\textup{-\underline{Gproj}}$.  
 
That above discussion provides examples of singular equivalences of Morita type between non-necessarily Gorenstein algebras that induce equivalences between stable categories of Gorenstein-projective modules. One question is to decide whether universal deformation rings of Gorenstein-projective modules are stable under such equivalences. With that inquiry in mind, our goal in this note is to prove the following result. 

\begin{theorem}\label{thm1}
Let $\A$ and $\G$ be finite dimensional $\k$-algebras and assume that ${_\G}X_\A$ and ${_\A}Y_\G$ are bimodules that induce a singular equivalence of Morita type with level  between $\A$ and $\G$ (as in Definition \ref{defi:3.2}) and which also induce  equivalences of triangulated categories $X\otimes_\A-:\A\textup{-\underline{Gproj}}\to \G\textup{-\underline{Gproj}}$ and $Y\otimes_\G-:\G\textup{-\underline{Gproj}}\to \G\textup{-\underline{Gproj}}$ that are quasi-inverse of each other. If $V$ is an indecomposable Gorenstein-projective left $\A$-module with $\SEnd_\A(V)\cong \k$, then $X\otimes_\A V$ is a finitely generated Gorenstein-projective left $\G$-module such that  $\SEnd_\G(X\otimes_\A V)\cong \k$ and the universal deformation rings $R(\A,V)$ and $R(\G, X\otimes_\A V)$ are isomorphic. 
\end{theorem}

It is important to mention that there are examples of finite dimensional $\k$-algebras $\A$ and $\G$ and bimodules ${_\G}X_\A$ and ${_\A}Y_\G$ that induce a singular equivalence of Morita type with level but do not induce an equivalence between $\A\textup{-\underline{Gproj}}$ and $\G\textup{-\underline{Gproj}}$ (see e.g. \cite[Example 7.5]{skart}).

Based on the discussion above, we obtain the following immediate consequence from Theorem \ref{thm1}.  

\begin{corollary}\label{cor3.1}
Let $\A$ and $\G$ be finite dimensional $\k$-algebras and assume that ${_\G}X_\A$ and ${_\A}Y_\G$ are bimodules that induce a singular equivalence of Morita type with level  between $\A$ and $\G$ (as in Definition \ref{defi:3.2}) such that $\Hom_\A(Y,\A)$ and $\Hom_\G(X,\G)$ are of finite dimension as a left $\G$-module and as a left $\A$-module, respectively. If $V$ is a finitely generated Gorenstein-projective indecomposable left $\A$-module with $\SEnd_\A(V)\cong \k$, then $X\otimes_\A V$ is a finitely generated Gorenstein-projective left $\G$-module such that  $\SEnd_\G(X\otimes_\A V)\cong \k$ and the universal deformation rings $R(\A,V)$ and $R(\G, X\otimes_\A V)$ are isomorphic. 
\end{corollary}

Let $\A$, $\G$, ${_\G}X_\A$ and ${_\A}Y_\G$ as in Theorem \ref{thm1}. If $\A$ and $\G$ are Gorenstein, then it follows from \cite[Lemma 4.12]{chen-liu-wang} that $\Hom_\G(X,\G)$ and $\Hom_\A(Y,\A)$ have finite projective dimensions as a left $\A$-module and a left $\G$-module, respectively. Therefore, Corollary \ref{cor3.1} generalizes \cite[Thm. 1.2 (ii)]{velez4}. 

We also have the following consequence of Corollary \ref{cor3.1}, which also generalizes \cite[Prop. 3.2.6]{blehervelez2} from self-injective to more general finite dimensional $\k$-algebras. 

\begin{corollary}\label{cor3.2}
Let $\A$ and $\G$ be finite dimensional $\k$-algebras and assume that ${_\G}X_\A$ and ${_\A}Y_\G$ are bimodules that induce a stable equivalence of Morita type $\A$ and $\G$ (in the sense of \cite{broue}) and such that $X$ and $Y$ do not have any projective bimodules as direct summands. If $V$ is a finitely generated Gorenstein-projective indecomposable left $\A$-module with $\SEnd_\A(V)\cong \k$, then $X\otimes_\A V$ is a finitely generated Gorenstein-projective left $\G$-module such that  $\SEnd_\G(X\otimes_\A V)\cong \k$ and the universal deformation rings $R(\A,V)$ and $R(\G, X\otimes_\A V)$ are isomorphic. 
\end{corollary}

\begin{example}
Let $\A=\k Q/\langle \rho\rangle$ and $\Gamma=\k Q'/\langle \rho'\rangle$ be the $\k$-algebras whose quivers with relations are given as follows:
\begin{align*}
Q&: \xymatrix@1@=20pt{\underset{1}{\bullet}\ar@/^/[r]^{\alpha}&\underset{2}{\bullet}\ar@/^/[l]^{\beta}},&& \rho=\{\beta\alpha\beta\alpha\},\\
Q'&: \xymatrix@1@=20pt{\underset{1'}{\bullet}\ar@/^/[r]^{x}&\underset{2'}{\bullet}\ar@/^/[l]^{y}\ar@(ur,dr)^{z}},&& \rho'=\{yx,zx,yz,z^2-xy\}.
\end{align*}
By \cite[\S 5]{liu-xi}, $\A$ and $\Gamma$ are stably equivalent of Morita type and their self-injective dimensions on both sides are equal to $2$. One easily verifies that both of them are not self-injective and have infinite global dimension. It was shown in \cite[Example 5.4]{bekkert-giraldo-velez} that there is a unique non-projective Gorenstein-projective indecomposable left $\A$-module $V$,  and a unique non-projective Gorenstein-projective indecomposable left $\G$-module $W$ such that $W$ corresponds (up to addition of projective modules) to $V$ under the stable equivalence of Morita type such that $\SEnd_\A(V)\cong \k \cong \SEnd_\G(W)$ and that the universal deformation rings $R(\A,V)$ and $R(\G,W)$ are both isomorphic to $\k[\![t]\!]/(t^2)$. This verifies Corollary \ref{cor3.2}
\end{example}

Finally, we also have the following consequence from Theorem \ref{thm1} concerning Morita context algebras.

\begin{corollary}\label{cor3.3}
Let $\A$, $\G$, ${_\G}X_\A$ and ${_\A}Y_\G$ be as in Corollary \ref{cor3.1}, and let $\Psi$ and $\Xi$ be as in \eqref{mc1} and \eqref{mc2}, respectively. Let $({_\Psi}\widetilde{X}_{\Xi}, {_\Xi}\widetilde{Y}_{\Psi})$ be bimodules that induce a singular equivalence of Morita type with level $\ell$ between $\Psi$ and $\Xi$. 
\begin{enumerate}
	\item If $\widetilde{V}$ is an indecomposable Gorenstein-projective left $\Xi$-module such that $\SEnd_\Xi(\widetilde{V})\cong \k$, then the left $\Psi$-module $\widetilde{X}\otimes_{\Xi} \widetilde{V}$ is also Gorenstein-projective with $\SEnd_{\Psi}(\widetilde{X}\otimes_{\Xi} \widetilde{V})\cong \k$ and the universal deformation rings $R(\Xi, \widetilde{V})$ and $R(\Psi, \widetilde{X}\otimes_{\Xi} \widetilde{V})$ are isomorphic in $\widehat{\Ca}$.    
	\item Assume further that the $\k$-algebra $C$ in \eqref{mc1} and \eqref{mc2} has finite projective dimension both as a $\Psi$-$\Psi$-bimodule as well as a $\Xi$-$\Xi$-bimodule and that $\A$ and $\G$ are both Gorenstein. 
	\begin{enumerate}
		\item There  exists a pair of bimodules $({_\Xi }\overline{X}_\A, {_\A}\overline{Y}_{\Xi})$ (resp.  $({_\Psi}\overline{U}_\G, {_\G}\overline{Z}_{\Psi})$) that induce a singular equivalence of Morita type with level between $\A$ and $\Xi$ (resp. between $\G$ and $\Psi$). 
		\item Let $V$ an indecomposable Gorenstein-projective left $\A$-module with $\SEnd_\A(V)\cong \k$. Then the left modules $V'= \overline{X}\otimes_\A V$, $V''=\widetilde{X}\otimes_\Xi\overline{X}\otimes_\A V$ and $V''' = \overline{Z}\otimes_{\Psi}\widetilde{X}\otimes_{\Xi}\overline{X}\otimes_\A V$, are all Gorenstein-projective with stable endomorphism ring isomorphic to $\k$ and the universal deformation rings $R(\A,V)$, $R(\Psi,V')$, $R(\Xi, V'')$ and $R(\G,V''')$ are all isomorphic in $\widehat{\Ca}$. 
	\end{enumerate}
\end{enumerate}	
\end{corollary}

\section{Preliminaries}\label{sec2}
Recall that $\k$ denotes a field of any characteristic. We denote by $\widehat{\Ca}$ the category of all complete local commutative Noetherian $\k$-algebras with residue field $\k$. In particular, the morphisms in $\widehat{\Ca}$ are continuous $\k$-algebra homomorphisms that induce the identity map on $\k$.  Let $\A$ be a fixed finite dimensional $\k$-algebra, and let $R$ be a fixed but arbitrary object in $\widehat{\Ca}$.  We denote by $R\A$ the tensor product of $\k$-algebras $R\otimes_\k\A$. Note that if $R$ is an Artinian object in $\widehat{\Ca}$, then $R\A$ is also Artinian (both on the left and the right sides).   If $R$ is an Artinian object in $\widehat{\Ca}$, then we denote by $\Omega_{R\A} M$ the first syzygy of $M$, i.e. $\Omega_{R\A} M$ is the kernel of a projective cover $P\to M$ of $M$ over $R\A$, which is unique up to isomorphism. Let $V$ be a finitely generated left $\A$-module. A {\it lift} $(M,\phi)$ 
of $V$ over $R$ is a finitely generated left $R\A$-module $M$ 
that is free over $R$ 
together with an isomorphism of $\A$-modules $\phi:\k\otimes_RM\to V$. Two lifts $(M,\phi)$ and $(M',\phi')$ over $R$ are {\it isomorphic} 
if there exists an $R\A$-module 
isomorphism $f:M\to M'$ such that $\phi'\circ (\mathrm{id}_\k\otimes_R f)=\phi$.
If $(M,\phi)$ is a lift of $V$ over $R$, we  denote by $[M,\phi]$ its isomorphism class and say that $[M,\phi]$ is a {\it deformation} of $V$ 
over $R$. We denote by $\Def_\A(V,R)$ the 
set of all deformations of $V$ over $R$. The {\it deformation functor} corresponding to $V$ is the 
covariant functor $\widehat{\Fun}_V:\widehat{\Ca}\to \Sets$ defined as follows: for all objects $R$ in $\widehat{\Ca}$, define $\widehat{\Fun}_V(R)=\Def_
\A(V,R)$, and for all morphisms $\theta:R\to 
R'$ in $\widehat{\Ca}$, 
let $\widehat{\Fun}_V(\theta):\Def_\A(V,R)\to \Def_\A(V,R')$ be defined as $\widehat{\Fun}_V(\theta)([M,\phi])=[R'\otimes_{R,\theta}M,\phi_\theta]$, 
where $\phi_\theta: \k\otimes_{R'}
(R'\otimes_{R,\theta}M)\to V$ is the composition of $\A$-module isomorphisms 
\[\k\otimes_{R'}(R'\otimes_{R,\theta}M)\cong \k\otimes_RM\xrightarrow{\phi} V.\] 


Suppose there exists an object $R(\A,V)$ in $\widehat{\Ca}$ and a deformation $[U(\A,V), \phi_{U(\A,V)}]$ of $V$ over $R(\A,V)$ with the 
following property. For all objects $R$ in $\widehat{\Ca}$ and for all deformations $[M,\phi]$ of $V$ over $R$, there exists a morphism $\psi_{R(\A,V),R,[M,\phi]}:R(\A,V)\to R$ 
in $\widehat{\Ca}$ such that 
\[\widehat{\Fun}_V(\psi_{R(\A,V),R,[M,\phi]})[U(\A,V), \phi_{U(\A,V)}]=[M,\phi],\]
and moreover, $\psi_{R(\A,V),R,[M,\phi]}$ is unique if $R$ is the ring of dual numbers $\k[\epsilon]$ with $\epsilon^2=0$.  Then $R(\A,V)$ and $
[U(\A,V),\phi_{U(\A,V)}]$ are called the {\it versal deformation ring} and {\it versal deformation} of $V$, respectively. If the morphism $
\psi_{R(\A,V),R,[M,\phi]}$ is unique for all $R\in\Ob(\widehat{\Ca})$ and deformations $[M,\phi]$ of $V$ over $R$, then $R(\A,V)$ and $[U(\A,V),\phi_{U(\A,V)}]$ are 
called the {\it universal deformation ring} and the {\it universal deformation} of $V$, respectively.  In other words, the universal deformation 
ring $R(\A,V)$ represents the deformation functor $\widehat{\Fun}_V$ in the sense that $\widehat{\Fun}_V$ is naturally isomorphic to the $\Hom$ 
functor $\Hom_{\widehat{\Ca}}(R(\A,V),-)$. 

We denote by $\Fun_V$  the restriction of $\widehat{\Fun}_V$ to the full subcategory of Artinian objects in $\widehat{\Ca}$. Following \cite[\S 2.6]{sch}, we call the set $\Fun_V(\k[\epsilon])$ the tangent space of $\Fun_V$, which has a structure of a $\k$-vector space by \cite[Lemma 2.10]{sch}. 
It was proved in \cite[Prop. 2.1]{blehervelez} that $\Fun_V$ satisfies the Schlessinger's criteria \cite[Thm. 2.11]{sch}, that there exists an isomorphism of $\k$-vector spaces 
\begin{equation}\label{hoch}
\Fun_V(\k[\epsilon])\to \Ext_\A^1(V,V),
\end{equation}
and that $\widehat{\Fun}_V$ is continuous in the sense of \cite[\S 14]{mazur}, i.e. for all objects $R$ in $\widehat{\Ca}$, we have  
\begin{equation}\label{cont}
\widehat{\Fun}_V(R)=\invlim_n \Fun_V(R/\m_R^n), 
\end{equation}
where $\m_R$ denotes the unique maximal ideal of $R$. Consequently, $V$ has always a well-defined versal deformation ring $R(\A,V)$ which is also universal provided that $\End_\A(V)$ is isomorphic to $\k$. It was also proved in \cite[Prop. 2.5]{blehervelez} that versal deformation rings are invariant under Morita equivalences between finite dimensional $\k$-algebras.

\begin{remark}\label{rem2.1}
\begin{enumerate}
\item It follows from the isomorphism of $\k$-vector spaces (\ref{hoch}) that if $\dim_\k \Ext_\A^1(V,V)=r$, then the versal deformation ring $R(\A,V)$ is isomorphic to a quotient algebra of the power series ring  $\k[\![t_1,\ldots,t_r]\!]$ and $r$ is minimal with respect to this property. In particular, if $V$ is a left $\A$-module such that $\Ext_\A^1(V,V)=0$, then $R(\A,V)$ is universal and isomorphic to $\k$ (see \cite[Remark 2.1]{bleher15} for more details).

\item Because of the continuity of the deformation functor as in (\ref{cont}), most of the arguments concerning $\widehat{\Fun}_V$ can be carried out for $\Fun_V$, and thus we are able to restrict ourselves to discuss liftings of $\A$-modules over Artinian objects in $\widehat{\Ca}$.

\item Let $R$ be an Artinian ring in $\widehat{\Ca}$, let $\iota_R:\k\to R$ be the unique morphism in $\widehat{\Ca}$ endowing $R$ with a $\k$-algebra structure, and let $\pi_R:R\to \k$ be the natural projection in $\widehat{\Ca}$. Then $\pi_R\circ \iota_R =\mathrm{id}_\k$. For all projective left (resp. right) $\A$-modules $P$, we let $P_R=R\otimes_{\k,\iota_R}P=R\otimes_\k P$. Then $P_R$ is a projective left (resp. right) $R\A$-module cover of $P$, and $(P_R, \pi_{P,R})$ is a lift of $P$ over $R$, where $\pi_{P,R}$ is the natural isomorphism $\k\otimes_{R, \pi_R}P_R\to P$.

\item Assume that $V$ is an indecomposable Gorenstein-projective left $\A$-module with $\SEnd_\A(V)=\k$.

\begin{enumerate}
\item  It follows by \cite[Thm. 1.2 (i)]{bekkert-giraldo-velez} that a deformation $[M,\phi]$ of $V$ over $R$ in $\widehat{\Ca}$ does not depend on the particular choice of the $\A$-module isomorphism $\phi$. More precisely, if $f:M\to M'$ is an $R\A$-module isomorphism with $(M',\phi')$ a lift of $V$ over $R$, then there exists an $R\A$-module isomorphism $\bar{f}:M\to M'$ such that $\phi'\circ (\mathrm{id}_\k\otimes _R\bar{f})=\phi$, i.e., $[M,\phi]=[M',\phi']$ in $\widehat{\Fun}_V(R)=\Def_\A(V,R)$. Moreover, by \cite[Thm. 1.2 (ii)]{bekkert-giraldo-velez} the versal deformation ring is $R(\A,V)$ is universal.
\item If $P$ is a projective left $\A$-module, then it follows from \cite[Thm. 1.2 (ii)]{bekkert-giraldo-velez} that the versal deformation ring $R(\A,V\oplus P)$ is universal and isomorphic to $R(\A,V)$. This result follows from the fact that for all Artinian objects $R$ in $\widehat{\Ca}$, there is a bijection of set of deformations
\begin{align}
\tau_{V\oplus P,R}: \Fun_{V}(R) \to \Fun_{V\oplus P}(R)
\end{align}
which for all lifts $(M,\phi)$ of $V$ over $R$, $\tau_{P,R}([M,\phi])=[M\oplus P_R, \phi\oplus \pi_{P,R}]$, where $(P_R,\pi_{P,R})$  is as in (iii.a).

\item  Let $\alpha: P(V)\to V$ be a projective left $\A$-module cover of $V$ (which is unique up to isomorphism), and let $\Omega_\A V=\ker\alpha$. Then we obtain a short exact  sequence of left $\A$-modules 
\begin{equation*}\label{projcover}
0\to \Omega_\A V\xrightarrow{\beta}P(V)\xrightarrow{\alpha} V\to 0.
\end{equation*}

Let $(M,\phi)$ be a lift of $V$ over $R$. Since $P_R(V)=R\otimes_{\k, \iota_R}P(V)$ is a projective left  $R\A$-module cover of $P(V)$ by (iii), and  since $\alpha$ is an essential epimorphism, there exists an epimorphism of $R\A$-modules $\alpha_R: P_R(V)\to M$ such that $\phi\circ (\mathrm{id}_\k\otimes\alpha_R)= \alpha\circ \pi_{P(V),R}$. Moreover, it follows by \cite[Claim 1]{bleher15} that $\alpha_R: P_R(V)\to M$ is a projective left $R\A$-module cover of $M$. Let  $\Omega_{R\A}M:=\ker \alpha_R$. Note that since $M$ and $P_R(V)$ are both free over $R$, then $\Omega_{R\A}M$ is also free over $R$, and that there exists an isomorphism of left $\A$-modules $\Omega_{R\A}(\phi):\k\otimes_R\Omega_{R\A}M\to \Omega_\A V$ such that $\pi_{P(V),R}\circ (\mathrm{id}_\k\otimes\beta_R)=\beta\circ \Omega_{R\A}(\phi)$, where $\beta:\Omega_\A V\to P(V)$ and $\beta_R:\Omega_{R\A}M\to P_R(V)$ are the natural inclusions. In particular, $(\Omega_{R\A}M, \Omega_{R\A}(\phi))$ is a lift of $\Omega_\A V$ over $R$. As noted in \cite[Rem. 3.3]{velez4}, it follows from the proof of \cite[Thm. 1.1 (i)]{velez4}, that that for all $i\geq 1$, there is a bijection of set of deformations 
\begin{align}
\tau_{\Omega^iV,R}: \Fun_{V}(R) \to \Fun_{\Omega^iV}(R),
\end{align}
which is natural with respect to morphisms between Artinian objects in $\widehat{\Ca}$ and which for all lifts $(M,\phi)$ of $V$ over $R$,  $\tau_{\Omega^iV,R}([M,\phi])=[\Omega^i_{R\A}M, \Omega^i_{R\A}(\phi)]$.
\end{enumerate}

\item Let $\Gamma$ be another finite dimensional $\k$-algebra, and assume that there exists $\ell\geq 0$ and a pair of bimodules $({_\Gamma}X_\A,{_\A}Y_\Gamma)$ that induce a singular equivalence of Morita type with level $\ell$ between $\A$ and $\Gamma$ as in Definition \ref{defi:3.2}. 
\begin{enumerate}
\item There exist projective bimodules ${_\Gamma} Q_\Gamma$, ${_\Gamma} Q'_\Gamma$, ${_\A}P_\A$, and ${_\A}P'_\A$ such that.

\begin{enumerate}
\item $X\otimes_\A Y\oplus Q'\cong \Omega_{\Gamma^e}^\ell\Gamma\oplus Q$ as $\Gamma$-$\Gamma$-bimodules;
\item $Y\otimes_\Gamma X\oplus P'\cong \Omega_{\A^e}^\ell\A\oplus P$ as $\A$-$\A$-bimodules.
\end{enumerate}
\noindent
Then by tensoring at both sides of (ii) with $V$ over $\A$, we obtain an isomorphism of left $\A$-modules
\begin{equation*}
Y\otimes_\Gamma X\otimes_\A V\oplus (P'\otimes_\A V)\cong \Omega_\A^\ell V\oplus T'_\ell\oplus (P\otimes_\A V),
\end{equation*}
where $T'_\ell$, $P'\otimes_\A V$ and $P\otimes_\A V$ are projective left $\A$-modules. This implies that  $\Omega_\A^\ell V$ and $Y\otimes_\Gamma X\otimes_\A V$ are isomorphic  indecomposable objects in $\A$-\underline{Gproj}. In particular, $\Omega_\A^\ell V$ does not have projective direct summands, and thus by the Krull-Schmidt-Azumaya Theorem and by (i), we obtain that there exists a finitely generated projective left $\A$-module $T_\ell$ such that there is an isomorphism of left $\A$-modules
\begin{equation}\label{eqn2.9}
Y\otimes_\Gamma X\otimes_\A V\cong \Omega_\A^\ell V\oplus T_\ell. 
\end{equation} 

Similarly, if $W$ is an indecomposable Gorenstein-projective left $\Gamma$-module, it follows that there exists a projective left $\Gamma$-module $S_\ell$ such that there exists an isomorphism of left $\Gamma$-modules
\begin{equation}\label{eqn2.10}
X\otimes_\A Y\otimes_\Gamma W\cong \Omega_\Gamma^\ell W\oplus S_\ell. 
\end{equation} 
\item Let $R$ be a fixed Artinian object in $\widehat{\Ca}$. It follows that $X_R=R\otimes_\k X$ is projective as a left $R\Gamma$-module and as a right $R\A$-module, and $Y_R=R\otimes_\k Y$ is projective 
as a left $R\A$-module and as a right $R\Gamma$-module, and both are free over $R$. Note also that $X_R\otimes_{R\A}Y_R\cong R\otimes_\k(X\otimes_\A Y)$ as $R\Gamma$-$R\Gamma$-bimodules and $Y_R\otimes_{R\Gamma} X_R\cong R\otimes_\k (Y\otimes_\A X)$ as $R\A$-$R\A$-bimodules. Moreover, we also have that 
\begin{align*}
Y_R\otimes_{R\Gamma} X_R\oplus P'_R&\cong \Omega_{R\A^e}^\ell(R\A) \oplus P_R&&\text{ as $R\A$-$R\A$-bimodules, and}\\
X_R\otimes_{R\A} Y_R\oplus Q'_R&\cong \Omega_{R\A^e}^\ell(R\Gamma)\oplus Q_R&&\text{ as $R\Gamma$-$R\Gamma$-bimodules},
\end{align*}
where $P_R=R\otimes_\k P$ and $P'_R=R\otimes_\k P'$ (resp. $Q_R=R\otimes_\k Q$ and $Q'_R=R\otimes_\k Q'$) are projective  $R\A$-$R\A$-bimodules (resp. $R\Gamma$-$R\Gamma$-bimodules).
\end{enumerate}

\end{enumerate}
\end{remark}



\section{Proof of Theorem \ref{thm1}}
 
In the following, assume further that  $\A$, $\Gamma$, $V$, $\ell$, ${_\Gamma}X_\A$, ${_\A}Y_\Gamma$, $P$, $P'$, $Q$ , and $Q'$ are all as in Remark \ref{rem2.1}  and let $R$ be a fixed Artinian object in $\widehat{\Ca}$. Let $(M,\phi)$ be a lift of $V$ over $R$. Since $V$ is assumed to be indecomposable, we obtain by using Remark \ref{rem2.1} (v.b) that there is an isomorphism of $R\A$-$R\A$-bimodules
\begin{equation}\label{eqn3.11}
Y_R\otimes_{R\Gamma}X_R\otimes_{R\A} M\cong \Omega_{R\A}^\ell M\oplus (T_\ell)_R, 
\end{equation}
where $T_\ell$ is as in (\ref{eqn2.9}). Note also that $X_R\otimes_{R\A} M$ is free over $R$, and that there exists an isomorphism of left $\Gamma$-modules $\phi_{X_R\otimes_{R\A} M}: \k\otimes_R(X_R\otimes_{R\A} M)\to X\otimes_\A V$.  Thus we can define a morphism between sets of deformations 
\begin{equation}
\tau_{X\otimes_\A V, R}: \Fun_V(R)\to\Fun_{X\otimes_\A V}(R) 
\end{equation}
as follows. For all deformations $[M,\phi]$ of $V$ over $R$, let $\tau_{X\otimes_\A V, R}([M,\phi])=[X_R\otimes_{R\A}M,\phi_{X_R\otimes_{R\A}M}]$. Let $(M,\phi)$ and $(M',\phi')$ be lifts of $V$ over $R$ such that $[M,\phi]=[M',\phi']$ in $\Fun_V(R)$. It follows that there is an isomorphism of left $R\Gamma$-modules $g: X_R\otimes_{R\A}M\to X_R\otimes_{R\A}M'$. Note that by hypothesis, $X\otimes_\A V$ is an indecomposable non-projective Gorenstein-projective left $\Gamma$-module with $\SEnd_{\Gamma}(X\otimes_\A V)=\k$. Thus by Remark \ref{rem2.1} (iv.a), we obtain that $[X_R\otimes_{R\A} M, \phi_{X_R\otimes_{R\A} M}]=[X_R\otimes_{R\A} M', \phi_{X_R\otimes_{R\A} M'}]$ in $\Fun_{X\otimes_\A V}(R)$. This proves that $\tau_{X\otimes_\A V,R}$ is well-defined. Next assume that $[N,\varphi]$ is a deformation of $X\otimes_\A V$ over $R$. If we let $L=Y_R\otimes_{R\Gamma} N$, then $L$ is free over $R$ and there is a composition of isomorphisms between left $\A$-modules which we denote by $\phi_L$ and which is given as follows:
\begin{align*}
\k\otimes_RL= \k\otimes_R(Y_R\otimes_{R\Gamma} N)&\cong (\k\otimes_R Y_R)\otimes_\Gamma (\k\otimes_R N)\\ 
&\cong Y\otimes_\Gamma (X\otimes_\A V)\\
&\cong \Omega_{\A}^\ell V\oplus T_\ell,
\end{align*}
where the last isomorphism follows from (\ref{eqn2.9}). In particular, $(L,\phi_L)$ is a lift of $\Omega_\A^\ell V\oplus T_\ell$ over $R$. Thus
by Remark \ref{rem2.1} (iv.b,iv.c), there exists a lift $(L',\phi_{L'})$ of $\Omega_\A^\ell V$ over $R$ such that $L'\oplus (T_\ell)_R\cong L$ as left $R\A$-modules. On the other hand, by Remark \ref{rem2.1} (iv.c), there exists a lift $(L'',\phi_{L''})$ of $V$ over $R$ such that $\Omega_{R\A}^\ell L''\cong L'$ as left $R\A$-modules. Therefore, $Y_R\otimes_{R\Gamma} N\cong \Omega_{R\A}^\ell L''\oplus (T_\ell)_R$ as left $R\A$-modules.  Note that we also have that $Y_R\otimes_{R\Gamma}X_R\otimes_{R\A} L''\cong \Omega_{R\A}^\ell L'\oplus (T_\ell)_R$ as left $R\A$-modules. Thus we obtain an isomorphism of left $R\A$-modules $Y_R\otimes_{R\Gamma} N\cong Y_R\otimes_{R\Gamma}X_R\otimes_{R\A} L''$, which induces a composition of isomorphisms between left $R\Gamma$-modules as follows:
\begin{align*}
\Omega_{\Gamma}^\ell N\oplus (S_\ell)_R &\cong X_R\otimes_{R\A}Y_R\otimes_{R\Gamma} N\\
&\cong X_R\otimes_{R\A}Y_R\otimes_{R\Gamma}X_R\otimes_{R\A} L''\\
& \cong \Omega_{R\Gamma}^\ell(X_R\otimes_{R\A} L'')\oplus (S_\ell)_R,
\end{align*}
where $S_\ell$ is as in (\ref{eqn2.10}). Thus $\Omega_{R\Gamma}^\ell N\cong \Omega_{R\Gamma}^\ell (X_R\otimes_{R\A} L'')$ as left $R\Gamma$-modules. By Remark \ref{rem2.1}(iv.c) (applied to $X\otimes_\A V$), we obtain that $N\cong X_R\otimes_{R\A}L''$ as left $R\Gamma$-modules and that $[N,\varphi]= [X_R\otimes_{R\A}L'',\phi_{X_R\otimes_{R\A}L''}]$ in $\Fun_{X\otimes_\A V}(R)$. This proves that $\tau_{X\otimes_\A V, R}$ is surjective. Next assume that $[M,\phi]$ and $[M',\phi']$ are deformations of $V$ over $R$ such that $[X_R\otimes_{R\A}M,\phi_{X_R\otimes_{R\A}M}]=[X_R\otimes_{R\A}M',\phi_{X_R\otimes_{R\A}M'}]$. In particular, assume that there exists an isomorphism of left $R\Gamma$-modules $f: X_R\otimes_{R\A}M\to X_R\otimes_{R\A}M'$. Thus we obtain an isomorphism of left $R\A$-modules $\mathrm{id}_{Y_R}\otimes f: Y_R\otimes_{R\Gamma}X_R\otimes_{R\A} M\to Y_R\otimes_{R\Gamma}X_R\otimes_{R\A} M'$. By (\ref{eqn3.11}), we obtain that there exists an isomorphism of left $R\A$-modules $\Omega_{R\A}^\ell f: \Omega_{R\A}^\ell M\to \Omega_{R\A}^\ell M'$ such that $\mathrm{id}_\k \otimes {\Omega_{R\A}^\ell f} =\mathrm{id}_{\Omega_\A^\ell V}$. This together with Remark \ref{rem2.1} (iv.a) implies that $[\Omega_{R\A}^\ell M,\Omega_{R\A}^\ell \phi]= [\Omega_{R\A}^\ell M',\Omega_{R\A}^\ell\phi']$ in $\Fun_{\Omega_\A^\ell V}(R)$, which together with Remark \ref{rem2.1} (iv.c) implies that $[M,\phi] = [M',\phi']$ in $\Fun_V(R)$. This proves that $\tau_{X\otimes_\A V, R}$ is injective. Next assume that $\theta: R\to R'$ is a morphism of Artinian objects in $\widehat{\Ca}$. Let $(M,\phi)$ a lift of $V$ over $R$. Then there is a composition of left $R'\A$-modules as follows:
\begin{align*}
R'\otimes_{R,\theta} (X_R\otimes_{R\A} M)\cong (R'\otimes_{R,\theta}X_R)\otimes_{R'\otimes_{R,\theta}R\A}(R'\otimes_{R,\theta}M)\cong X_{R'}\otimes_{R'\A}M',
\end{align*} 
where $X_{R'}=R'\otimes_\k X$ and $M'= R'\otimes_{R,\theta} M$. This proves that $\tau_{X\otimes_\A V, R}$ is natural with respect of morphism between Artinian objects in $\widehat{\Ca}$. 

The continuity of the deformation functor (see Remark \ref{rem2.1} (ii)) implies that for all objects $R$ in $\widehat{\Ca}$, there is a bijection between sets of deformations
\begin{equation*}
\widehat{\tau}_{X\otimes_\A V, R}: \widehat{\Fun}_V(R)\to \widehat{\Fun}_{X\otimes_\A V}(R),
\end{equation*}
which is natural with respect of morphisms between objects in $\widehat{\Ca}$. Consequently, we obtain that the universal deformation rings $R(\A,V)$ and $R(\A,X\otimes_\A V)$ are isomorphic in $\widehat{\Ca}$. This finishes the proof of Theorem \ref{thm1} (ii). 

\begin{proof}[Proof of Corollary \ref{cor3.3}]
Let $\A$, $\G$, ${_\G}X_\A$, ${_\A}Y_\G$, $\Psi$, $\Xi$, ${_\Psi}\widetilde{X}_{\Xi}$ and  ${_\Xi}\widetilde{Y}_{\Psi}$ be as in in the hypothesis of Corollary \ref{cor3.3}. Statement (i) follows from \cite[Theorem (2), pg. 239]{gao-ma-wang} and Theorem \ref{thm1}. Assume next that $\A$ and $\G$ are both Gorenstein and that $C$ has finite-projective dimension both as a $\Xi$-$\Xi$-bimodule and as a $\Psi$-$\Psi$-bimodule. Thus the statement (ii.a) follows from \cite[Example 4.6]{dalezios}. Moreover,  it follows from \cite[Cor. 4.6 \& Thm. 4.13]{gaopsa} that $\Xi$ and $\Psi$ are also Gorenstein. Thus statement (ii.b) follows from \cite[Thm. 1.2 (ii)]{velez4} and (ii.a).   
\end{proof}

\section{Acknowledgements.} The original idea of this paper was motivated when the first author held a postdoctoral position at the Bishop's University in 2013. He would like to thank Professor Thomas Br\"{u}stle for his hospitality and useful discussions. The first author was also supported by the Fundamental Research Funds for the Central Universities 2024JBMC001 from Beijing Jiaotong university and partially supported by Beijing Natural Science Foundation (1252011). The second author was supported by the research group CIMI in the Centro de Investigaciones at the Fundaci\'on Universitaria Konrad Lorenz, and by the Office of Academic Affairs at the Valdosta State University.

\bibliographystyle{amsplain}
\bibliography{Morita_Type_with_Level-Rev}

\end{document}